\documentclass[a4paper,12pt]{article}

\usepackage{amsmath,amssymb,amstext,amsfonts,amsthm,latexsym}
\usepackage{lineno,hyperref}
\usepackage{dsfont}
\usepackage{pdfsync,color}
\newtheorem{theorem}{Theorem}
\newtheorem{lemma}{Lemma}[section]
\newtheorem{proposition}{Proposition}[section]

\newtheorem{corollary}{Corollary}[section]
\newtheorem{remark}{Remark}[section]
\newtheorem{example}{Example}[section]

\makeatother

\newcommand{\CC}{\mathds{C}}

\newcommand{\logp}{\log^+}
\newcommand{\PP}{\mathds{P}}

\newcommand{\PPn}[1]{\mathds{P}_{#1}}

\newcommand{\RR}{\mathds{R}}
\newcommand{\ZZ}{\mathds{Z}}
\newcommand{\ZZp}{\mathds{\mathds{Z}}_{+}}
\newcommand{\Lp}[1]{\mathbf{L}^p\left(#1 \right)}
\newcommand{\Lq}{\mathbf{L}^q}
\newcommand{\Nzero}{{\mathbf{N}}_z}
\newcommand{\Nsign}{{\mathbf{N}}_{\rm o}}
\newcommand{\reg}{\mathbf{Reg}}
\newcommand{\ch}[1]{\mbox{\bf{Co}}\left(#1 \right)}
\newcommand{\capa}[1]{\mathbf{cap}\left(#1 \right)}

\newcommand{\inter}[1]{\stackrel{\circ}{#1}}
\newcommand{\intch}[1]{\mathop{\stackrel{\circ}{\mbox{\bf Co}}\left(#1 \right)}}
\newcommand{\sopor}[1]{\mathrm{supp}\left(#1 \right)}
\newcommand{\sgn}[1]{{\mathrm{sgn}\/}\left(#1 \right)}
\newcommand{\dgr}[1]{\mathrm{deg}\left(#1\right)}
\newcommand{\dsty}{\displaystyle}

\newcommand{\imm}[1]{{\mathop{Im}\/}\left(#1 \right)}
\newcommand{\ree}[1]{{\mathrm{Re}\/}\left(#1 \right)}
\DeclareRobustCommand{\wlim}{\mathop{\operatorname{w{-}lim}}}

\newcommand{\D}{\mathbf{D}}
    \makeatletter
    \renewcommand*{\@fnsymbol}[1]{\ensuremath{\ifcase#1\or *\or        \mathsection\or\natural\or\dagger\or \ddagger\or
  \mathparagraph\or **\or \dagger\dagger
        \or \ddagger\ddagger \else\@ctrerr\fi}}
    \makeatother
\title{Sobolev extremal polynomials with respect to mutually singular measures}

\author{A. D\'{\i}az Gonzalez\thanks{abdiazgo@math.uc3m.es} \thanks{Supported by the Research Fellowship Program, Ministerio de Econom\'{\i}a, Industria y Competitividad  of Spain,  under grant  MTM2015-65888-C4-2-P.}  \and \and G. L\'{o}pez Lagonasino\thanks{lago@math.uc3m.es} $^{\ddag}$ \and H. Pijeira Cabrera\thanks{hpijeira@math.uc3m.es} \thanks{Research partially supported by  Ministerio de Econom\'{\i}a, Industria y Competitividad  of Spain, under grant  MTM2015-65888-C4-2-P.} \\ \emph{Universidad Carlos III de Madrid,  Spain}}

\date{}
\begin{document}

\maketitle

\begin{abstract}
We consider extremal polynomials with respect to a Sobolev-type $p$-norm, with $1<p<\infty$ and measures supported on compact subsets of the real line. For a wide class of such  extremal polynomials with respect to  mutually singular measures (i.e. supported on disjoint subsets of the real line), it is proved that their critical points are simple and contained in the interior of the convex hull of the support of the measures involved  and the asymptotic  critical point distribution is studied. We also find the $n$th root asymptotic behavior of the corresponding sequence of Sobolev extremal polynomials and their derivatives.

\begin{description}
  \item[{\sl Keywords:}] extremal polynomials, Sobolev  orthogonality, location of zeros, asymptotic behavior
  \item[{\sl AMS Subject classification:}] 42A05, 30C15,  	26C05,  26C10,  33C47
\end{description}
\end{abstract}

\section{Introduction}

Let $\mu_0$  be a positive Borel measure supported on an interval of the real line $\Delta_0$ (which does not reduce to a point).
  For $1 < p < \infty$, we denote by $\Lp{\mu_0}$ the Banach space of all $p$-integrable functions on $\Delta_0$ with respect to the measure $\mu_0$, endowed with the norm
\begin{equation}\label{Norm-Lp}
   \|f\|_{0,p}=\left(\int_{\Delta_0} |f|^p d\mu_0\right)^{1/p}.
\end{equation}

We denote by $\PPn{n}$ the space of polynomials (with complex coefficients) of degree $\leq n\,,$  and  by $\PP^*_n\subset \PPn{n}$ the subset of monic polynomials of exact degree $n$. It is well known  that $\|\cdot\|_{0,p}$ is a \emph{strictly convex norm}, i.e.  the unit ball is a strictly convex set. Then, there exists a unique monic polynomial $P_n \in \PP^*_n$  such that
\begin{equation}\label{ExtremalPol_Lp}
   \|P_n\|_{0,p} \,=\, \min_{Q\in\PP^*_{n}} \|Q\|_{0,p},
\end{equation}
(cf. \cite[Def. 7.5.1 \& Th. 7.5.3]{Dav75}). $P_n$ is called the  \emph{$n$th monic extremal polynomial } relative to $ \|\cdot\|_{0,p}$.

The study of zeros and critical points of extremal polynomials is of {great interest because they can be interpreted in various ways} from the standpoint of physics,  function theory and numerical analysis. It is known that the norm \eqref{Norm-Lp}
is a \emph{Fej\'{e}r norm} (i.e., for distinct  $f,g \in \Lp{\mu_0}$ the condition  $ |f(z)|\leq |g(z)|$ for all $z\in \Delta_0$, with equality
only if $g(z)\equiv 0$,  implies $\|f\|_{0,p}< \|g\|_{0,p}$). Hence, from \emph{Fej\'{e}r's convex hull theorem} \cite[Th. 10.2.2]{Dav75}
we get that the zeros of $P_n$ are simple and lie in $\Delta_0$.

Let us mention a characterization of the solution of the extremal problem  \eqref{ExtremalPol_Lp} (cf. \cite[\S2.2, Ex. 7-h]{BorErd95}). A polynomial $P_n \in \PP^*_n$ is  the  $n$th monic extremal polynomial  in $\Lp{\mu_0}$ if and only if
 for all $Q \in \PP_{n-1}$
\begin{equation}\label{ThChar-Lp}
   \int_{\Delta_0} Q \,\sgn{P_n}\,\left|P_n\right|^{p-1} \;d\mu_0 = 0,
    \text{ where } \sgn{y} = \left\{
             \begin{array}{ll}
              {y}/{|y|}, & \hbox{if }  y\neq 0;\\
               0, & \hbox{if } y=0.
             \end{array}
           \right.
\end{equation}

Hence, if $P_n$ has a zero of multiplicity at least two at $x^*$ then, $\frac{P_n(x)}{(x-x^*)^2}$ is a polynomial of degree $(n-2)$ and we have the contradiction
$$
0 < \int_{\Delta_0}  \;\frac{\left|P_n(x)\right|^{p}}{(x-x^*)^2} \,d\mu_0 (x) = \int_{\Delta_0
} \sgn{P_n(x)}\, \left|P_n(x)\right|^{p-1} \;\frac{P_n(x)}{(x-x^*)^2} \,d\mu_0 (x)=0.
$$
Consequently \emph{all the zeros of $P_n$ are simple}.

Let $\mu_0, \,\mu_1$ be two  positive Borel measures supported on the intervals 
$\Delta_0\subset \RR$ and $\Delta_1\subset \RR$ respectively, where $\Delta_0$ is an non trivial interval. For  
$1<p<\infty$, we consider on the space $\PP$ of polynomials, the Sobolev norm
\begin{equation}\label{SobolevNorm_p}
   \|f\|_{S,p}=\left(\|f\|_{0,p}^p+\|f^{\prime}\|_{1,p}^p\right)^{\frac{1}{p}}=\left(\int_{\Delta_0} |f|^p d\mu_0+\int_{\Delta_1} |f^{\prime}|^p d\mu_1.\right)^{\frac{1}{p}}
\end{equation}
It is not difficult to prove that \eqref{SobolevNorm_p} is a strictly convex norm and, therefore, for each $n\in \ZZp$  there exists a unique monic polynomial $L_n \in \PP^*_n$  such that
\begin{equation}\label{ExtremalPol-Sob}
   \|L_n\|_{S,p} \,=\, \min_{Q\in\PP^*_{n}} \|Q\|_{S,p}.
\end{equation}
The polynomial $L_{n}$ is called the $n$th  \emph{monic extremal polynomial} relative to  $\|\cdot\|_{S,p}$. In Proposition \ref{Propo_ExUn} we give an alternative direct proof of the uniqueness of  $L_{n}$. Obviously, $\|\cdot\|_{S,p}$  is not a Fej\'{e}r norm because we  can construct (piecewise continuously differentiable) functions such that $|f(z)| < |g(z)|, x \in \Delta_0 \cup \Delta_1,$ with $|f^\prime|$ much larger than  $|g^\prime|$,  $\mu_1$ a.e. on $\Delta_1$, so that $\|f\|_{S,p} > \|g\|_{S,p}$. Specific examples are easy to produce.
\begin{example}
Let $\mu_0, \mu_1$ be probability measures supported on $\Delta_0$ and $\Delta_1$, respectively. Take $0 < a < c < 1$ and assume that $\Delta_0\cup \Delta_1 \subset [-a,a]$. If $f(x) = x$ and $g(x) \equiv c$, then $|f(x)| < |g(x)|, x \in \Delta_0\cup \Delta_1$, whereas
\[\|g\|_{S,p}^p = c^p < 1 < \int_{\Delta_0} |x|^p d\mu_0(x) +1 = \|f\|_{S,p}^p.\]
\end{example}

It is well known that in the standard case of extremality with respect to the norm \eqref{Norm-Lp} the   zeros of the extremal polynomial are all in $\Delta_0$ 
and \eqref{Norm-Lp} is a Fej\'{e}r norm, but in the Sobolev case this is not true. P. Althammer shows in an early example (cf. \cite{Alt62}, where $p=2$) that in the Sobolev case the zeros of the orthogonal polynomials may lie outside of 
$\Delta_0\cup \Delta_1$. Other examples of the previous fact can be seen in \cite[\S 2, where $p=2$]{GaKu97}.

However, in the numerical experiments carried out in \cite[\S 2]{GaKu97} (for $p=2$), the authors found that in all the cases considered, the critical points of $L_n$ were real numbers. Their experiments conclude with two conjectures about the location of zeros and critical points of the Sobolev-type orthogonal  polynomials (see \cite[Conjectures 1 and 2]{GaKu97}). In the following theorem,  we solve the problem derived from these conjectures for extremal polynomials when $\mu_0$ and $\mu_1$ in \eqref{SobolevNorm_p} are supported on mutually disjoint intervals.

\begin{theorem}\label{theo-funda}
Let $p\in (1,\infty)$ and let $\mu _0,\mu _1$ be finite positive Borel measures supported on  the real line
 such that $\inter{\Delta_0}\cap\inter{\Delta_1}=\emptyset$ ($\inter{A}$ denotes the interior of a real set $A$ with the Euclidean topology of $\RR$). 
  Then
\begin{itemize}

\item[\ref{theo-funda}.1.] for all $n\geq 1$, $\dsty   n-1 \leq \Nsign(L_{n};\inter{\Delta_0}) + \Nsign(L_{n}^{\prime};\inter{\Delta_1}) \leq n\;$ and the zeros of $L_{n}$ in $\inter{\Delta_0}$ are simple, where  the symbol  $\Nsign(Q;I)$ denotes the number of zeros with odd multiplicity of the polynomial $Q\in \PP$  on the interval $I\subset \RR$ (i.e. points of sign change).
  \item[\ref{theo-funda}.2.] for $n\geq 2$, the critical points of the extremal polynomial $L_{n}$ are simple and contained in $\intch{\Delta_0 \cup \Delta_1}$. ($\ch{A}$ denotes the convex hull of the set $A$).
\item[\ref{theo-funda}.3.] the number of  zeros (or critical points) of $L_{n}$ lying in
$\dsty \intch{\Delta_0 \cup \Delta_1} \setminus \left(\inter{\Delta_0} \cup \inter{\Delta_1}\right)$ is at most one. 
\item[\ref{theo-funda}.4.] the zeros of $L_{n}^{\prime}$ in $\inter{\Delta_0}$ 
 interlace the zeros of $L_{n}$ on that set.
\end{itemize}

\end{theorem}

The next result provides a natural and intrinsic characterization of the extremal polynomials defined by \eqref{ExtremalPol-Sob}, and an extension of \eqref{ThChar-Lp} for the Sobolev case.  

\begin{theorem}\label{ThChar-1}
 $L_{n}$ is the $n$th monic extremal polynomial with respect to $\|\cdot\|_{S,p}$ if and only if
 \begin{eqnarray*}
 \langle Q, L_n \rangle_{S,p} &:=&   \int_{\Delta_0} Q(x)\,\sgn{L_{n}(x)}\,|L_{n}(x)|^{p-1}d\mu_0(x) \nonumber  \\ & &  +
\int_{\Delta_1} Q^{\prime}(x)\,\sgn{L_{n}^{\prime}(x)}\,|L_{n}^{\prime}(x)|^{p-1}d\mu_1(x)=0,
\end{eqnarray*}
 for every polynomial $Q \in \PPn{n-1}$.
\end{theorem}

Note that unless $p=2$, $ \langle \cdot, \cdot \rangle_{S,p}$ does not define an inner product.

Theorem \ref{ThChar-1} is a corollary of Theorem \ref{ThChar-2} below, where  \eqref{SobolevNorm_p} is replaced by a Sobolev norm with derivatives  of higher order. More precisely, for $f\in\PP$, set
\begin{equation}\label{SobolevNorm_p-m}
   \|f\|_{S,p,m}=\left(\sum_{k=0}^{m} \|f^{(k)} \|_{k,p}^p\right)^{\frac{1}{p}}=\left(\sum_{k=0}^{m}\int |f^{(k)}|^p d\mu_k\right)^{\frac{1}{p}},
\end{equation}
where $m$ is a fixed non-negative integer, $1<p<\infty$, $\mu_k$ is a positive Borel measure supported on $\RR$ ($k=0,\dots,m$),  $\sopor{\mu_0}$ is an infinite set, 
and $f^{(k)}$ denotes the $k$th derivative of $f$. When $m=1$,    \eqref{SobolevNorm_p-m}, reduces to \eqref{SobolevNorm_p} and $\|\cdot\|_{S,p}=\|\cdot\|_{S,p,1}$.

According to  \eqref{SobolevNorm_p-m}, $L_n \in \PP^*_n$ (the $n$th monic extremal polynomial with respect to \eqref{SobolevNorm_p-m}) is a monic polynomial   that verifies
\begin{equation}\label{ExtremalPoly_p-m}
   \|L_n\|_{S,p,m} \,=\, \min_{Q\in\PP^*_{n}} \|Q\|_{S,p,m}.
\end{equation}

When $p=2$, and the norm \eqref{SobolevNorm_p-m} is given by an inner product, the corresponding Sobolev extremal polynomials (or orthogonal with respect to the associated inner product) have been extensively studied. A survey on the subject is provided in \cite{MaXu15}. However, for $p\neq 2$ ($1<p<\infty$) not much has been attained and the basic references are \cite{LoPePi05}  for the so called \emph{``sequentially dominated  norms''} and \cite{LopMarPij06} for measures with unbounded support on the real line.

Section 2 is devoted to the study of the existence and uniqueness of the extremal polynomial with respect to the norm \eqref{SobolevNorm_p-m}. Theorem  \ref{ThChar-2}, which is of independent interest, is the main tool for locating zeros and critical points. In Section \ref{Sec-2Intervals}, we prove Theorem \ref{theo-funda} and Corollary \ref{Theo_ZerCrit} on the location and algebraic properties of the zeros and critical points of the extremal polynomials $L_n$.
The rest of the paper is devoted to the study of the asymptotic zeros distribution of the zeros and critical points of the Sobolev extremal polynomials.  Next, we state the main result in this direction after introducing some needed terminology.

For any complex polynomial $Q_{n}(z)=c \prod_{k=1}^{n} (z-z_k)$, with $c,z_1,\ldots,z_n \in \CC$,  we denote by $\sigma(Q_{n})$ the so called \emph{normalized zero counting measure associated with} $Q_{n}$, as
\begin{equation}\label{ZeroCountMeasure-1}
 \sigma(Q_{n}) = \frac{1}{n} \sum_{j=1}^{n} \delta_{z_j} \;,
\end{equation}
where  $\delta_{z_j}$ is the Dirac measure with mass one at the point $z_j$. Following the usual terminology, if $\{\mu_n\}$ is a sequence of measures on a compact set $K \subset \mathbb{C}$, we say that a measure $\mu$ is the limit  of $\{\mu_n\}$ in  the weak star  topology of measures, if
$$\lim_{n\to \infty}\int f d\mu_n=\int f d\mu,$$
for every continuous function $f$ on $K$. In this case we write $\dsty \wlim_{n\to \infty} \mu_n=\mu$.

Let $\mu$ be a finite Borel measure whose compact support $S(\mu) \subset \mathbb{C}$ has positive logarithmic capacity and let $P_n$ be the associated monic orthogonal polynomial with respect to $\mu$ of degree $n$. We say that $\mu$ is regular and write $\mu \in \reg$ if
\[\lim_{n \to \infty} \|P_n\|_{\mu,2}^{1/n} =  \capa{S(\mu)},\]
where   $\|P_n\|_{\mu,2}$ denotes the $L_2(\mu)$ norm of $P_n$. Theorem 3.1.1 in \cite{StaTo92} contains several equivalent forms of defining regular measures (see also \cite[Theorems 3.2.1, 3.2.3]{StaTo92}). Recall that for any compact set $K \subset \mathbb{C}$ with $\capa{K} > 0$ there exists a unique probability measure $\mu_K, S(\mu_K) \subset K,$ called the \emph{equilibrium measure} of $K$, which is characterized by
\[ \int \log\frac{1}{|z-x|} d\mu_K(x)
\left\{
\begin{array}
{cl}
= \gamma, & z \in K \setminus A, \,\,\capa{A} =0, \\
\leq \gamma, & z \in \mathbb{C},
\end{array}\right.\]
where $A$ is a Borel set, and $\gamma$ is some uniquely determined constant (actually $e^{-\gamma} = \capa{K}$).

\begin{theorem}\label{ThRegAD} Let $\{L_{n}\}$ be the sequence of monic extremal polynomials relative to  \eqref{SobolevNorm_p} with $p \in [1,\infty)$ and  $\mu_0,\mu_1 \in \reg$. Then, for each integer $j> 0$
\begin{align}\label{Thm-AsinDer-1}\lim_{n \to\infty}   \|L_{n}^{(j)}\|_{\Delta}^{1/n} & = \capa{\Delta},
\\ \label{Thm-AsinDer-2}
  \wlim_{n \to \infty} \sigma\left(L_{n}^{(j)}\right)&=\mu_{\Delta},
\end{align}
where $\Delta = \Delta_0 \cup \Delta_1$ and $\mu_{\Delta}$ is the equilibrium measure on $\Delta$.
\end{theorem}

Notice that \eqref{Thm-AsinDer-2} holds for $j > 0$. For $j=0$ the zeros of the polynomials $L_n$ can abandon $\Delta$ and their asymptotic zero distribution is governed by the balayage of $\mu_\Delta$ onto a certain region which we describe later (for details, see Theorem \ref{BalayageTh01} below).

\section{The Characterization Theorem}

Throughout this section   we consider the more general Sobolev norm \eqref{SobolevNorm_p-m} and $L_n$ verifies \eqref{ExtremalPoly_p-m}. As $\PPn{n-1}$ is a finite dimensional linear space, the existence of $L_n \in \PP^*_n$ is obvious.
In addition, $L_n$ has real coefficients, for otherwise $L_n$ could be rewritten as $L_n=P+iQ$ where $Q$ and $P$ are polynomials with real coefficients, $Q\not\equiv 0$ and $P$ is a monic polynomial of degree $n$ satisfying
\begin{align*}
\|L_n\|_{S,p,m}^p&=\sum_{k=0}^{m}\int |P^{(k)}+iQ^{(k)}|^p d\mu_k=\sum_{k=0}^{m}\int\left(\left(P^{(k)}\right)^2+\left(Q^{(k)}\right)^2\right)^{\frac{p}{2}}d\mu\\&> \sum_{k=0}^{m}\int\left| P^{(k)}\right|^pd\mu=\|P\|_{S,p,m}^p
\end{align*}

The next proposition contains a direct proof of the uniqueness of  the Sobolev extremal polynomial $L_{n}$ and shows that \eqref{SobolevNorm_p-m} is a strictly convex norm.

\begin{proposition}[Uniqueness] Let $\|\cdot\|_{S,p,m}$ be the Sobolev type norm defined by \eqref{SobolevNorm_p-m}. Then,  there exists a unique monic polynomial $L_{n}$ ($\dgr{L_{n}}=n$) such that $\dsty \|L_{n}\|_{S,p,m}= \inf_{Q_n \in\PP^*_{n}} \|Q_n\|_{S,p,m}.$
\label{Propo_ExUn}
\end{proposition}

\begin{proof} If $L_{n}$ and $\widetilde{L}_{n}$ are two different monic extremal polynomials of degree $n$,  from the extremality and the triangular inequality it is obvious that $\frac{1}{2}\left( L_{n}+\widetilde{L}_{n}\right)$ is also a monic extremal polynomials.
Hence
\begin{align}\label{triangular- equality}
\| L_{n}+\widetilde{L}_{n}\|_{S,p,m}=\|L_{n}\|_{S,p,m}+\|\widetilde{L}_{n}\|_{S,p,m}.
\end{align}
From the Minkowski inequality we obtain
\begin{eqnarray*}
\|L_{n}+\widetilde{L}_{n}\|_{S,p,m} &=& \left(\sum_{k=0}^{m} \|L_{n}^{(k)}+\widetilde{L}_{n}^{(k)} \|_{k,p}^p \right)^{1/p}\\    &\leq &
 \left(\sum_{k=0}^{m} \left(\|L_{n}^{(k)}\|_{k,p}+\|\widetilde{L}_{n}^{(k)}\|_{k,p}\right)^p \right)^{1/p}\\
    &\leq & \|L_{n}\|_{S,p,m}+\|\widetilde{L}_{n}\|_{S,p,m}.
\end{eqnarray*}
Therefore, by \eqref{triangular- equality}, the first inequality just shown is in fact the equality
$$\dsty \sum_{k=0}^{m} \|L_{n}^{(k)}+\widetilde{L}_{n}^{(k)} \|_{k,p}^p = \sum_{k=0}^{m}
\left(\|L_{n}^{(k)}\|_{k,p}+\|\widetilde{L}_{n}^{(k)}\|_{k,p}\right)^p.$$ However
$$\|L_{n}^{(k)}+\widetilde{L}_{n}^{(k)} \|_{k,p}^p\leq\left(\|L_{n}^{(k)}\|_{k,p}+\|\widetilde{L}_{n}^{(k)}\|_{k,p}\right)^p, \quad k=0,1,\cdots,m; $$
therefore,
$$\|L_{n}^{(k)}+\widetilde{L}_{n}^{(k)}\|_{k,p}  = \|L_{n}^{(k)}\|_{k,p}+\|\widetilde{L}_{n}^{(k)}\|_{k,p},\quad k=0,1,\cdots,m. $$
In particular, from the Minkowski inequality for $\Lp{\mu_0}$,  we have that there exists a constant $\alpha \geq 0$  such that $L_{n}=\alpha\widetilde{L}_{n}$ almost everywhere with respect to $\mu_0$. But, $L_{n}$ and $\widetilde{L}_{n}$ are monic polynomials and $\sopor{\mu_0}$ is an infinite set,  hence $L_{n}=\widetilde{L}_{n}$.
\end{proof}

\begin{theorem}[Characterization] \label{ThChar-2}  Let $\|\cdot\|_{S,p,m}$ $(1 < p < \infty)$ be the Sobolev type norm defined in   \eqref{SobolevNorm_p-m}. Then, the monic polynomial $L_{n}$ is the $n$th  \emph{monic extremal polynomial relative to}  $\|\cdot\|_{S,p,m}$ if and only if
\begin{equation}\label{ThChar-2-0}
\langle Q, L_n\rangle_{S,p,m} :=   \sum_{k=0}^m\int Q^{(k)}\sgn{L_n^{(k)}}\left|L_n^{(k)}\right|^{p-1}d\mu_k=0,
\end{equation}
 for every polynomial $Q \in \PPn{n-1}$. \label{ThChar-p}
\end{theorem}
\begin{proof}
 Assume that $L_{n}$ is the $n$th  \emph{monic extremal polynomial relative to}  $\|\cdot\|_{S,p,m}$  and let $Q \in \PPn{n-1}$, then
\begin{equation}\label{ThChar-2-1}
   \|L_{n}\|_{S,p,m}\leq \|L_{n}+ \alpha \,Q\|_{S,p,m}, \quad \mbox{ for all } \alpha \in \RR.
\end{equation}
Let $F(\alpha)$ be the auxiliary function defined for all $\alpha \in \RR$  by the expression
$$F(\alpha)=\|L_{n}+\alpha Q\|^p_{S,p,m}=\sum_{k=0}^m\int \left|L_{n}^{(k)}+\alpha Q^{(k)}\right|^pd\mu_k.$$
From Proposition \ref{Propo_ExUn} and \eqref{ThChar-2-1}, $\alpha=0$ is the unique minimum point of $F$, thus
\begin{eqnarray*}
0 = F^\prime(0)=  p \, \sum_{k=0}^m\int Q^{(k)}\sgn{L_n^{(k)}}\left|L_n^{(k)}\right|^{p-1}d\mu_k=p \, \langle Q, L_n\rangle_{S,p,m}
\end{eqnarray*}
and we get \eqref{ThChar-2-0}.
Now, assume that \eqref{ThChar-2-0} takes place for every polynomial $Q \in \PPn{n-1}$.  Obviously, each monic polynomial $\widetilde{Q}$ of degree $n$ can be written as the sum $\widetilde{Q}=L_{n}+Q$ where $Q \in \PPn{n-1}$.\\
Let $q$  be the conjugate exponent of $p$, i.e. $q= \frac{p}{p-1}$. For $k=0,1,\cdots ,m$ we have
\begin{eqnarray*}
G_k= \sgn{L^{(k)}_{n,p}}\,|L^{(k)}_{n,p}|^{p-1} \in \Lq(\mu_k),\quad \int\left|G_k\right|^{q} d\mu_k= \|L_{n}\|_{k,p}^{p},
\end{eqnarray*}
Thus
\begin{align}\label{ThChar-2-2}
\|L_{n}\|_{S,p,m}^p = \sum_{k=0}^m\|L_{n}\|_{k,p}^{p}=\sum_{k=0}^m\int \left|G_k\right|^{q} d\mu_k=\sum_{k=0}^m\|G_k\|_{k,q}^q.
\end{align}
Let $\alpha,\beta \geq 0$, $p>1$ and $q= \frac{p}{p-1}$. It is well known that
\begin{equation}\label{ThChar-2-3}
    \sqrt[p]{\alpha}\cdot  \sqrt[q]{\beta} \leq \frac{\alpha}{p}+ \frac{\beta}{q},
\end{equation}
with equality if and only if $\alpha=\beta$ (cf. \cite[\S 2.1.1--Thm. 2]{MiPe70}).
From \eqref{ThChar-2-0}, H\"{o}lder's inequality, \eqref{ThChar-2-3} and \eqref{ThChar-2-2}, we get
\begin{eqnarray*}
 \|L_{n}\|_{S,p,m}^p &=& \langle L_{n},L_{n}\rangle_{S,p,m}=\langle L_{n}+Q,L_{n}\rangle_{S,p,m}=\langle \widetilde{Q},L_{n}\rangle_{S,p,m},\\
   &=&  \sum_{k=0}^m\int \widetilde{Q}^{(k)}G_kd\mu_k,\leq \sum_{k=0}^m\|\widetilde{Q}^{(k)}\|_{k,p}\cdot \|G_k\|_{k,q},\\
&\leq& \sum_{k=0}^m\left(\frac{\|\widetilde{Q}^{(k)}\|_{k,p}^p}{p}+ \frac{\|G_k\|_{k,q}^q}{q}\right)= \frac{\|\widetilde{Q}\|_{S,p,m}^p}{p}+ \frac{\|L_{n}\|_{S,p,m}^p}{q}.
\end{eqnarray*}
Thus $ \|L_{n}\|_{S,p,m} \leq \|\widetilde{Q}\|_{S,p,m}$, which completes the proof.
\end{proof}

\begin{corollary} Under the assumptions of Theorem \ref{ThChar-p}, if $n\geq 1$ then  $L_{n}$  has at least one zero of odd multiplicity  $x_0 \in \intch{\sopor{\mu_0}}$.  \label{Cor-UnCero}
\end{corollary}
\begin{proof}  This is an immediate consequence of  $$    \int \sgn{L_{n}}\,|L_{n}|^{p-1} d\mu_0 =\langle 1, L_{n}  \rangle_{S,p,m}=0.$$
\end{proof}

\begin{corollary} Under the assumptions of Theorem \ref{ThChar-p}, if $n\geq 2$ then $L^{\prime}_{n}$  has at least one zero of odd multiplicity in $\intch{\sopor{\mu_0}\cup \sopor{\mu_1}}$. \label{Cor-UnPCrit}
\end{corollary}
\begin{proof} Suppose  that $L_{n}^{\prime}$  has no zeros of odd multiplicity on $\intch{\sopor{\mu_0} \cup \sopor{\mu_1}}$,
then $L_{n}$ is  monotone on $\intch{\sopor{\mu_0} \cup \sopor{\mu_1}}$. From Corollary \ref{Cor-UnCero}, $L_n$ has exactly one zero $x_0$  of odd multiplicity on $\intch{\sopor{\mu_0}}$, so  $\sgn{(x-x_0)L_{n}(x)}=\sgn{L_n^{\prime}(x)}=c$ is constant for all $x \in \intch{\sopor{\mu_0}}$ with $c=\pm 1$. Hence, by Theorem \ref{ThChar-p} we have
\begin{align*}
0& =\langle c(x-x_0), L_n \rangle_{S,p,m}\\ &=\int c(x-x_0)\sgn{L_{n}}\,|L_{n}|^{p-1} d\mu_0+\int c\cdot\sgn{L^{\prime}_{n}}\,|L^{\prime}_{n}|^{p-1} d\mu_1 > 0
\end{align*}
which is a contradiction.
\end{proof}

\section{Two disjoint intervals. Proof of  Theorem \ref{theo-funda}}\label{Sec-2Intervals}

Here, we prove the results on the location of zeros and critical points announced previously. The first lemma in this section is a consequence of Biernacki`s theorem \cite[Th. 4.5.2]{RahSch02}, which in turn is a converse of the Gauss-Lucas theorem.
\begin{lemma}[\textbf{Biernacki}]   \label{Lemma-Biernacki}\emph{\cite[Thm. 4.5.2]{RahSch02}} Let $K$ be the convex hull of the critical points of a polynomial $f$, and let $f(z_0)=0$. Then, the zeros of $f$ lie in the union of all the closed disks centered at the vertices  of $K$ and radius equal to the distance from the vertex to $z_0$.
\end{lemma}

Let $I\subset \RR$ be an interval and $Q \in \PP$. As in Theorem \ref{theo-funda}, $\Nsign(Q;I)$  denotes the number of zeros of $Q$ on $I$ with odd multiplicity (i.e. points of sign change). Additionally, $\Nzero(Q;I)$ denotes the total  number of zeros (counting multiplicities) of $Q$ on $I$ and for all  $n\geq 1$ we write $\ell_n:=\Nsign(L_{n};\inter{\Delta_0}) + \Nsign(L_{n}^{\prime};\inter{\Delta_1})$. The key to the proof of Corollary \ref{Theo_ZerCrit} is the following trivial consequence of Rolle`s theorem.

\begin{lemma}\emph{\cite[Lemma 2.1]{LoPiPe01}} \label{lemma-Rolle} Let $I$ be an interval of the real line and $Q$ a non-constant polynomial of degree $n$ with real coefficients, then  $\dsty  \Nzero(Q;I) + \Nzero(Q^{\prime};\CC \setminus I) \leq n. $
\end{lemma}


\begin{proof}[Proof of Theorem \ref{theo-funda}] For $n=1,2$ the statements of the Lemma are immediate consequences of Corollaries \ref{Cor-UnCero} and \ref{Cor-UnPCrit}.
So, in the sequel we assume that $n\geq 3$. From Lemma \ref{lemma-Rolle} we have $l_n\leq n$. The simplicity of the zeros of $L_n$ in $\inter{\Delta_0}$ follows directly from
the inequality $\ell_n \geq n-1$ and Lemma \ref{lemma-Rolle}. Therefore, to complete the proof of the statement 1.1 it suffices to show that
\begin{align}\label{4.1}
\ell_n \geq n-1.
\end{align}
Without loss of generality, we can assume that $\inter{\Delta_0} = (a,b)$, $\inter{\Delta_1} = (c,d)$ and $-\infty \leq a < b\leq c \leq d \leq \infty$, the case $d \leq a$ is solved similarly.

Fix $n \geq 3$ and let $x_0$ be the point in $(a,b)$ closest to $[c,d]$ where $L_{n}$ changes sign. This point exists due to Corollary \ref{Cor-UnCero}. There are two possible cases, either
\begin{equation}
\label{LocCriCaso1}\tag{I}
  \sgn{L_{n}^{\prime}(x_0+ \epsilon) \cdot L_{n}^{\prime}(c+ \epsilon)}=1
\end{equation}
for all sufficiently small $\epsilon > 0$, or
\begin{equation}
\label{LocCriCaso2}\tag{II}
\sgn{L_{n}^{\prime}(x_0+ \epsilon) \cdot L_{n}^{\prime}(c+ \epsilon)}=-1
\end{equation}
for all sufficiently small $\epsilon > 0$. Let us consider  each case separately.

In case \ref{LocCriCaso1} we can prove more than \eqref{4.1}; namely,
\begin{equation}
\label{lemma-4_1}
 \ell_n=n \;.
\end{equation}
To the contrary, suppose that $\ell_n\leq n-1$ in case \ref{LocCriCaso1} or $\ell_n\leq n-2$ in case \ref{LocCriCaso2}. We shall see that we can find a polynomial $Q\in\PPn{n-1}$ such that
\begin{align}\label{sign}
Q(x)L_{n}(x)\geq 0, \quad x\in [a,b] && \text{ and } && Q^{\prime}(x)L_{n}^{\prime}(x) \geq 0, \quad  x \in [c,d].
\end{align}

Suppose that \eqref{sign} holds, using Theorem \ref{ThChar-1} we get
\begin{align*}
0 &= \langle Q,L_{n} \rangle _{S,p,1}\\
&=\int_a^b Q\,\sgn{L_{n}}\,|L_{n}|^{p-1}d\mu_0
     +
\int_c^d Q^{\prime}\sgn{L_{n}^{\prime}}|L_{n}^{\prime}|^{p-1}d\mu_1>0,
\end{align*}
which is a contradiction and the proof of \textit{\ref{theo-funda}.1.} would be complete. Therefore, it is sufficient to find such a polynomial $Q$.

\medskip\noindent
\textbf{Case I}

\medskip
Suppose that $\ell_n\leq n-1$ and take $Q$ to be a polynomial of degree $\leq\ell_n$ with real coefficients, not identically equal to zero, which
has a zero at each point of $(a,b)$ where $L_{n}$ changes sign and whose derivative has a zero at each point of $(c,d)$ where $L_{n}^{\prime}$ changes sign. The existence of $Q$ reduces to solving a system of $\ell_n$ equations on $\ell_n+1$ unknowns (the coefficients of $Q$); thus, a non trivial solution always exists. Notice that
$$
\ell_n \leq \Nzero(Q;(a,b)) + \Nzero(Q^{\prime};(c,d))
$$
with strict inequality if either $Q$ (resp. $Q^{\prime}$) has on $(a,b)$ (resp. $(c,d)$) zeros of multiplicity greater than one or distinct
from those assigned by construction. On the other hand, because of Corollary \ref{Cor-UnCero} the degree of $Q$ is at least $1$; therefore, using Lemma \ref{lemma-Rolle}, we have that
$$
\ell_n \leq \Nzero(Q;(a,b)) + \Nzero(Q^{\prime};(c,d)) \leq \dgr{Q} \leq \ell_n.
$$
Thus
\begin{equation}\label{degQ}
\ell_n = \Nzero(Q;(a,b)) + \Nzero(Q^{\prime};(c,d)) = \dgr{Q} \;.
\end{equation}
Hence $Q$ (resp. $Q^{\prime}$) has on $(a,b)$ (resp. $(c,d)$) simple zeros and has no other zero different from those given by construction. So, $QL_{n}$ and $Q^{\prime}L_{n}^{\prime}$ have constant sign on $[a,b]$ and $[c,d]$, respectively. We can choose $Q$ in such a way that $QL_{n} \geq 0$ on $[a,b]$ (if this was not so replace $Q$ by $-Q$). Then, to prove \eqref{sign} it  remains to check that $\sgn{Q^{\prime}(c+\epsilon)L_{n}^{\prime}(x_0+\epsilon)}=1$ for all $\epsilon$ sufficiently small. From Rolle's Theorem and \eqref{degQ} we have
\begin{align*}
\ell_n-1&= \Nsign(Q;(a,x_0)) + \Nsign(Q^{\prime};(c,d)) \\
&\leq \Nsign(Q^{\prime};(a,x_0))+ \Nsign(Q^{\prime};(c,d)) \leq\ell_n-1
\end{align*}
Hence $\Nsign(Q^{\prime};(a,x_0)) + \Nsign(Q^{\prime};(c,d))=\dgr{Q^{\prime}}$ and all the zeros of $Q^{\prime}$ are contained in $(a,x_0) \cup (c,d)$. So, for all $\epsilon$ sufficiently small, we have
$$\sgn{Q^{\prime}(x_0+\epsilon)\cdot Q^{\prime}(c+\epsilon)}=1.$$
Now, from this expression  and \eqref{LocCriCaso1}, we obtain
\begin{align*}
\sgn{Q^{\prime}(c+\epsilon)L_{n}^{\prime}(c+\epsilon)}&=\sgn{Q^{\prime}(c+\epsilon)}\sgn{L_{n}^{\prime}(c+\epsilon)}\\
&=\sgn{Q^{\prime}(x_0+\epsilon)}\sgn{L_{n}^{\prime}(x_0+\epsilon)}\\
&=\sgn{Q(x_0+\epsilon)}\sgn{L_{n}(x_0+\epsilon)}\\
&=\sgn{Q(x_0+\epsilon)L_{n}(x_0+\epsilon)} =1.
\end{align*}
Therefore, we get \eqref{sign} and hence \eqref{lemma-4_1} ($\ell_n=n$).

In order to prove in this case the remaining statements, notice that
\begin{align*}
n-1&=l_n-1 = \Nsign(L_{n};(a,x_0)) + \Nsign(L_{n}^{\prime};(c,d))\\
   & \leq  \Nsign(L^{\prime}_{n};(a,x_0))+ \Nsign(L_{n}^{\prime};(c,d))+\Nsign(L_{n}^{\prime};(x_0,c])\leq n-1.
\end{align*}
Therefore, $\Nsign(L_{n}^{\prime};(a,d))=n-1$, $\Nsign(L_{n}^{\prime};(x_0,c])=0$, and $\Nsign(L_{n};(a,x_0))=\Nsign(L_{n}^{\prime};(a,x_0))$ which proves \textit{\ref{theo-funda}.2}, \textit{\ref{theo-funda}.3} and \textit{\ref{theo-funda}.4}.

\medskip\noindent
\textbf{Case II}

\medskip
Suppose that $\ell_n\leq n-2$. The difference consists in that to the right of $x_0$ and $c$ the polynomial
$L_{n}^{\prime}$ has different signs. Here, we construct $Q$ of degree $ \leq \ell_n+1\leq n-1$ with real coefficients, not identically equal to zero, with the same interpolation conditions as above plus $Q^{\prime}(c) =0$. Following the same line of reasoning, we have
$$
\ell_n+1 = \Nzero(Q;(a,b)) + \Nzero(Q^{\prime};[c,d)) = \dgr{Q} \;.
$$
Hence $Q$ (resp. $Q^{\prime}$) has on $(a,b)$ (resp. $[c,d)$) simple zeros and no other zero except
those given by construction. So $QL_{n}$ and $Q^{\prime}L_{n}^{\prime}$ have constant sign on $[a,b]$ and $[c,d]$, respectively. Analogous to the previous case, $\Nsign(Q^{\prime};(a,x_0)) + \Nsign(Q^{\prime};[c,d))=\dgr{Q^{\prime}}$ and all the zeros of $Q^{\prime}$ are contained in $(a,x_0) \cup [c,d)$. Now, using that $Q$ changes sign at $c$   and (\ref{LocCriCaso2}) we obtain
\begin{align*}
\sgn{Q^{\prime}(c+\epsilon)L_{n}^{\prime}(c+\epsilon)}&=\sgn{Q^{\prime}(c+\epsilon)}\sgn{L_{n}^{\prime}(c+\epsilon)}\\
&=-\sgn{Q^{\prime}(x_0+\epsilon)}\left(-\sgn{L_{n}^{\prime}(x_0+\epsilon)}\right)\\
&=\sgn{Q(x_0+\epsilon)}\sgn{L_{n}(x_0+\epsilon)}\\
&=\sgn{Q(x_0+\epsilon)L_{n}(x_0+\epsilon)}=1
\end{align*}
which proves that $Q$ satisfies \eqref{sign} and hence \eqref{4.1} is true.

Now, notice that \eqref{LocCriCaso2} and  the intermediate value theorem imply that  $L_{n}^{\prime}$ has at least an odd zero on the interval $(x_0,c]$, thus
\begin{align*}
n-1&\leq (\ell_n-1) +1   = \Nsign(L_{n};(a,x_0)) + \Nsign(L_{n}^{\prime};(c,d)) +1\\
    & \leq \Nsign(L_{n}^{\prime};(a,x_0))  + \Nsign(L_{n}^{\prime};(c,d))+\Nsign(L_{n}^{\prime};(x_0,c])\leq n-1
\end{align*}
Therefore, $\Nsign(L_{n}^{\prime};(a,d))=n-1$, $\Nsign(L_{n}^{\prime};(x_0,c])=1$, and $\Nsign(L_{n};(a,x_0))=\Nsign(L_{n}^{\prime};(a,x_0))$ from which  \textit{\ref{theo-funda}.2}, \textit{\ref{theo-funda}.3}, and \textit{\ref{theo-funda}.4} follow.
\end{proof}
\begin{corollary}
\label{Theo_ZerCrit}
Let $p\in (1,\infty)$ and let $\mu _0,\mu _1$ be finite positive Borel measures defined on the real line such that $\inter{\Delta_0}\cap\inter{\Delta_1}=\emptyset $
If we take $[a,b]=\Delta_0$ and $[c,d]=\Delta_1$, where $a,b,c,d \in \RR$, then the zeros of $L_{n}$ lie in
\begin{equation}\label{BiernackiCompact}
D_{\Delta}=\begin{cases}
\D(d,d-a)\cup \D(a,b-a), & \text{if } b \leq c,\\
\D(c,b-c)\cup \D(b,b-a), & \text{if } d \leq  a,\\
\end{cases}.
\end{equation}
where $ \D(a,r)=\{z\in\CC: |z-a|<r\}$.
\end{corollary}

\begin{proof}
For $n=1$ the statement is certainly true because of Corollary \ref{Cor-UnCero}. For $n \geq 2$ the result follows directly from \textit{\ref{theo-funda}.2} in Theorem \ref{theo-funda} and the Biernacki Lemma \ref{Lemma-Biernacki}.
\end{proof}

\section{Regular asymptotic  distribution of critical points}

A compact set $K$ of the complex plane  is said to be  \emph{regular} if the Green function with singularity at $\infty$ relative to the unbounded connected component of $\CC \setminus K$ can be extended continuously to the boundary. We refer the reader to \cite{Lan72,Ran95} and for short  \cite[Appendix]{StaTo92} for this and other notions related with logarithmic potential theory. For example, the union of a finite number of bounded intervals in the real line form regular compact sets.

Suppose that $\mu$ is a finite positive Borel measure such that $S(\mu)$ is a regular compact set and $1\leq p <\infty$. It is well known (see \cite[Thm. 3.4.3]{StaTo92}) that $\mu \in \reg$  if and only if
\begin{equation}\label{compnormas1}
\lim_{n \to\infty}  \left({\frac{\|Q_{n}\|_{S(\mu)}}{\|Q_{n}\|_{\mu,p}}}\right)^{1/n} = 1\,,
\end{equation}
where $\{Q_{n}\}, n \in \mathbb{Z}_+,$ is any sequence of polynomials such that $\dgr{Q_{n}} = n$  and  $\|\cdot\|_{S(\mu)}$ denotes the usual sup  norm on $S(\mu)$. Using Cauchy's integral theorem for the derivative of a holomorphic function, with the same hypothesis on $S(\mu)$ it is easy to show (see \cite[lemma 3]{LoPi99}) that for all $j \in \ZZp$
\begin{equation}\label{compnormas2}
\limsup_{n \to\infty}\left({\frac{\|Q_n^{(j)}\|_{S(\mu)}}{\|Q_n\|_{S(\mu)}}}\right)^{1/n} \leq 1.
\end{equation}

One last result that we will used is contained in \cite[Thm. 2.1 - Cor. 2.1]{BlSaSi88}. Let $K$ be a compact subset of the real line with $\capa{K}>0$ and let $\{Q_{n}\}$ be a sequence of monic polynomials, $\dgr{Q_{n}} = n$. Then
\begin{equation}\label{weaklim}
 \limsup_{n \to \infty} \|Q_n\|_{K}^{1/n}=  \capa{K} \qquad \Rightarrow  \qquad \wlim_{n \to\infty} \sigma(Q_{n}) =  \mu_K.
\end{equation}

\begin{proof}[Proof of Theorem \ref{ThRegAD}] From \eqref{weaklim}, \eqref{Thm-AsinDer-2} follows from \eqref{Thm-AsinDer-1}. Therefore, let us prove \eqref{Thm-AsinDer-1}.

Let $T_n$  denote the  $n$th monic Chebyshev polynomial with respect to $\Delta$; that is, $\|T_n\|_{\Delta} \leq \|Q_n\|_{\Delta}$ for any monic polynomial $Q_n$ of degree $n$. In particular,
\begin{align*}
\liminf_{n \to\infty}   \|L_{n}^{(j)}\|_{\Delta}^{1/n} & \geq \liminf_{n \to\infty}  \left(\frac{n!}{(n-j)!}\|T_{n-j}\|_{\Delta}\right)^{1/n}\\
& \geq \liminf_{n \to\infty}   \|T_{n-j}\|_{\Delta}^{1/n} = \capa{\Delta},
\end{align*}
since it is well known that $\lim_{n \to \infty} \|T_n\|_{\Delta}^{1/n}=\capa{\Delta}$ (see, for example, \cite[Cor. 5.5.5]{Ran95}).
Hence, we  only need to prove that
\begin{equation}\label{ThRegAD_CPoint-1}
\limsup_{n \to\infty}   \|L_{n}^{(j)}\|_{\Delta}^{1/n} \leq  \capa{\Delta}.
\end{equation}

Since  $\mu_0,\mu_1 \in \reg$ and $\Delta_0,\Delta_1$ are regular compact sets, from  \eqref{compnormas1} and \eqref{compnormas2} we obtain
 \begin{align}
    \limsup_{n \to\infty}   \| L_{n}^{(j)}\|_{\Delta_0}^{p/n}&\leq\limsup_{n \to\infty}   \| L_{n}\|_{\Delta_0}^{p/n} \leq \limsup_{n \to\infty}  \|L_{n}\|_{\mu_0,p}^{p/n} \nonumber \\ &  \leq  \limsup_{n \to\infty}   \|L_{n}\|_{S,p}^{p/n},\label{AsymRootSobNorm-01}
 \end{align}
 and
 \begin{align*}
 \limsup_{n \to\infty}   \| L_{n}^{(j)}\|_{\Delta_1}^{p/n} &\leq\limsup_{n \to\infty}   \| L_{n}^{\prime}\|_{\Delta_1}^{p/n} \leq \limsup_{n \to\infty}  \|L_{n}^{\prime}\|_{\mu_1,p}^{p/n} \\ & \leq \limsup_{n \to\infty}   \|L_{n}\|_{S,p}^{p/n}.
    \end{align*}
The conjunction of these two relations imply that
\begin{equation}
\label{cotanorm}
\limsup_{n \to\infty}   \| L_{n}^{(j)}\|_{\Delta}^{1/n} \leq \limsup_{n \to\infty}   \|L_{n}\|_{S,p}^{1/n}.
\end{equation}

Now, from the extremality of  $L_{n}$ in the Sobolev norm, we get
\begin{align*}
    \|L_{n}\|_{S,p}^p &\leq \|T_{n}\|_{S,p}^p = \|T_n\|_{\mu_0,p}^p+\|T_n^{\prime}\|_{\mu_1,p}^p, \leq \mu_0\left(\Delta_0\right)\, \|T_n\|_{\Delta_0}^p + \mu_1\left(\Delta_1\right)\, \|T_n^{\prime}\|_{\Delta_1}^p\\
  &
  \leq   \mu_0\left(\Delta_0\right)\, \|T_n\|_{\Delta}^p + \mu_1\left(\Delta_1\right)\, \|T_n^{\prime}\|_{\Delta}^p.
    \end{align*}
On the other hand, using again \eqref{compnormas1} and \eqref{compnormas2} it follows that
\begin{align*}
\dsty \limsup_{n \to \infty} \left(\frac{\|L_{n}\|_{S,p}}{\|T_n\|_{\Delta}}\right)^{1/n}\leq  \limsup_{n \to \infty} \left( \mu_0\left(\Delta_0\right) + \mu_1\left(\Delta_1\right)\,\frac{ \|T_n^{\prime}\|_{\Delta}^p}{\|T_n\|_{\Delta}^p}\right)^{ {1}/{pn}}\leq 1,
\end{align*}
whence
\begin{equation}\label{AsymRootSobNorm-02}
 \limsup_{n \to\infty}  \|L_{n}\|_{S,p}^{1/n} \leq  \capa{\Delta}.
\end{equation}
Now, \eqref{cotanorm} and \eqref{AsymRootSobNorm-02} give \eqref{ThRegAD_CPoint-1} and we are done.
\end{proof}

\begin{remark}
If ${\Delta_0}=[a,b]$ and  ${\Delta_1}=[c,d]$ are bounded and non trivial intervals of $\RR$, for the equilibrium measure $\mu_{\Delta}$ in Theorem \ref{ThRegAD}, we have two possibilities:
\begin{itemize}
  \item If $\Delta_0 \cap \Delta_1 \neq \emptyset$,  then $\Delta=[\alpha,\beta]$ (an interval) and $\mu_{\Delta}$ is the arcsine measure
   \begin{equation}\label{EqulMeasure_1}d\mu_{\Delta}= \frac{dx}{\pi \sqrt{|(x-\alpha)(x-\beta)|}}, \quad x \in (\alpha,\beta).\end{equation}
  \item If $\Delta_0 \cap \Delta_1 = \emptyset$ then there exists $x_* \in  \Delta_{g}=\ch{\Delta_0 \cup \Delta_1} \setminus  \left( \Delta_0 \cup \Delta_1\right)$ such that
\begin{equation}\label{EqulMeasure_x_*}
  \int_{\Delta_{g}} \frac{(x-x_*)\,dx}{\sqrt{|(x-a)(x-b)(x-c)(x-d)|}}=0,\;
\end{equation}
and (see \cite[Lemma 4.4.1]{StaTo92})
   \begin{equation}\label{EqulMeasure_2}
  d\mu_{\Delta}= \frac{|x-x_*|\,dx}{\pi \sqrt{|(x-a)(x-b)(x-c)(x-d)|}}, \quad x \in (a,b)\cup(c,d).
\end{equation}

\end{itemize}
\end{remark}

Combining Theorems \ref{theo-funda} and \ref{ThRegAD} we obtain

\begin{theorem}\label{ThDerAsympSqrt_n} Let $\mu_0, \mu_1 \in \reg$ be such that $\Delta_0$ and $\Delta_1$ are bounded and non trivial intervals of $\RR$ satisfying $\inter{\Delta}_0 \cap \inter{\Delta}_1 = \emptyset$. Denote by $\{L_{n}\}$  the sequence of monic extremal polynomials relative to  the corresponding Sobolev norm \eqref{SobolevNorm_p} with $p \in (1,\infty)$. Then, for all $j > 0$
\begin{equation}\label{CoroDerAsympSqrt_n-1}
    \limsup_{n\to \infty}  \left| L_{n}^{(j)}(z)\right|^{1/n}= \capa{\Delta} e^{g_{\Omega}(z;\infty)},\; z \in \CC
\end{equation}
except on a set of zero capacity, where $g_{\Omega}(z;\infty)$  is Green's function for $\Omega = \overline{\CC}\setminus \Delta$ with singularity at infinity (cf. \cite[Appendix V]{StaTo92}).

Moreover, uniformly on compact subset of $\widehat{\Omega}=  \CC \setminus \ch{\Delta_0 \cup \Delta_1}$
\begin{align}\label{CoroDerAsympSqrt_n-2}
    \lim_{n\to \infty}  \left| L_{n}^{(j)}(z)\right|^{1/n} & = \capa{\Delta} e^{g_{\Omega}(z;\infty)},
\end{align}
and
\begin{align}\label{CoroDerAsympSqrt_n-3}
 \lim_{n\to \infty} \frac{L_{n}^{(j+1)}(z)}{n L_{n}^{(j)}(z)}& = \int \frac{d\mu_{\Delta}(x)}{(z-x)},
\end{align}
where $x_*$ is given by \eqref{EqulMeasure_x_*}.
\end{theorem}

\begin{remark} To give an explicit expressions for the function on the right side of \eqref{CoroDerAsympSqrt_n-1}-\eqref{CoroDerAsympSqrt_n-2}, we assume without loss  of generality that $\Delta_0=[a,b]$ and $\Delta_1=[c,d]$ with $-\infty< a<b\leq c<d<\infty$. As we show below, there are closed formulas for Green's function $g_{\Omega}(z;\infty)$ and the logarithmic capacity of $\Delta$. When both segments have the same length, $g_{\Omega}(z;\infty)$ and $\capa{\Delta}$ are given by elementary formulas (see \eqref{GreenFun_Capa-2} below). In general, (formulas \eqref{GreenFun-1}-\eqref{Capa2Inter-1}), they can be expressed in terms of theta-functions (see \cite{Rum13,Shi76}) defined by
$$\theta(u,\tau,r,s)= \sum_{k\in \ZZ} e^{2\pi i \left((k+r)^2\frac{\tau}{2}+(k+r)(u+s)\right)}$$
where $u,\tau \in \CC$, $\imm{\tau}>0$ and $r,s \in \RR$. Here, we will be particularly interested in the functions
\begin{equation}\label{theta-function}
 \vartheta(u,\tau)=\theta\left(u,\tau,\frac{1}{2},\frac{1}{2}\right) \; \text{ and }\;   \vartheta_0(\tau)=\theta\left(0,\tau,0,\frac{1}{2}\right).
\end{equation}
Following the procedure for computing the logarithmic capacity of two segments given in \cite[\S 1.3.3]{FaSe01} and \cite[Ch. 2]{Rum13}, let $\Psi(z)$ be the function
$$ \Psi(z)= \sqrt{\frac{(z-a)(d-b)}{(z-b)(d-a)}}, $$
where $\sqrt{z}>0$ for $z>0$, and $\Upsilon(w)$ is the elliptic integral
$$\Upsilon(z)=\int_{0}^{z} \frac{dx}{\sqrt{(1-x^2)(1-\upsilon^2 x^2)}},\; \text{ where }\; \upsilon=\Psi(c)^{-1}.$$
Putting $\dsty \Phi(z)=\Upsilon\left( \Psi(z)\right)$ and $\dsty \tau=\frac{\Phi(c)}{\Phi(d)}$, we obtain
      \begin{align}\label{GreenFun-1}
   g_{\Omega}(z;\infty)& = - \log\left| \frac{\vartheta((2\Phi(d))^{-1}(\Phi(z)-\Phi(\infty)),\tau)}{\vartheta((2\Phi(d))^{-1}(\Phi(z)+\overline{\Phi(\infty)}),\tau)}\right|  \\
\label{Capa2Inter-1}  \text{ and }\; \capa{\Delta} & = \left|\frac{\vartheta_0(\tau)\;\sqrt[4]{(c-a)(c-b)(d-a)(d-b)}}{2 \; \vartheta\left(\Phi^{-1}(d) \ree{\Phi(\infty)},\tau\right)}\right|.
   \end{align}
We recall that  $a<b<c<d$. Hence if $-a=d$ and $-b=c$ for Green's function and the logarithmic capacity, we obtain
      \begin{equation}\label{GreenFun_Capa-2}
   g_{\Omega}(z;\infty)=  \frac{1}{2}\logp\left| \frac{\sqrt{z^2-b^2}+\sqrt{z^2-a^2}}{\sqrt{z^2-b^2}-\sqrt{z^2-a^2}}\right|  \;  \text{ and }\; \capa{\Delta}  = \frac{1}{2} \sqrt{a^2-b^2},
   \end{equation}
where $$\logp\left|z\right| =\left\{\begin{array}{ll} \log|z| & \hbox{ if } |z|>1  \\ 0  &  \hbox{ if } |z|\leq 1\end{array}.\right.$$
\end{remark}

\begin{proof}[Proof of Theorem \ref{ThDerAsympSqrt_n}.]  From Theorem \ref{theo-funda}, we have that
for all $n\geq 2$, the critical points of the extremal polynomial $L_{n}$ are simple and contained in $\intch{\Delta_0 \cup \Delta_1}$. Now, Rolle's theorem implies that the zeros of all derivatives of higher order of  $L_{n}$ lie in the convex hull of the
set of its critical points. Therefore, for all $n\geq 2$ and $j\geq 1$,  the $(n-j)$ zeros $\{x_k^{(j)}\}$ of $L_n^{(j)}$  lie on  $\intch{\Delta_0 \cup \Delta_1}$.  Thus, for each fixed  $j\geq 1$  the measure  $\sigma(L_n^{(j)})$ has its support contained in $\ch{\Delta_0 \cup \Delta_1}$. From the lower envelope theorem  (cf. \cite[Appendix  III]{StaTo92}) and
\eqref{Thm-AsinDer-2}, we get
\begin{equation}\label{ThDerAsympSqrt_n-P1}
  \liminf_{n \to \infty} \int \, \log \frac{1}{|z - x|}\, d\sigma(L_n^{(j)})(x) =
\int \, \, \log \frac{1}{|z - x|}\, d\mu_{\Delta}(x) \;,
\end{equation}
for all $z \in \CC$ except on a set of zero capacity. But, from \cite[Ch.1-(2.3)]{StaTo92}
$$ \int \, \log \frac{1}{|z - x|}\, d\mu_{\Delta}(x)  = \log \frac{1}{\capa{\Delta}}-g_{\Omega}(z;\infty)\;, $$
hence   \eqref{ThDerAsympSqrt_n-P1}  is equivalent to \eqref{CoroDerAsympSqrt_n-1}.

In order to prove \eqref{CoroDerAsympSqrt_n-2}, notice that
for each fixed $j\geq 0$, the family of functions
$$ \left\{ \int \, \log \frac{1}{|z - x|}\, d\sigma(L_n^{(j)})(x) \right\}
\;,\;\; n \in \ZZ_+ \;, $$
is harmonic and uniformly bounded on compact subsets of $\widehat{\Omega}$. From \eqref{CoroDerAsympSqrt_n-1}, any subsequence which converges uniformly on compact subsets of $\widehat{\Omega}$ must tend
to $\int \, \log |z - x|^{-1} \, d\mu_{\Delta}(x)$
(independent of the convergent subsequence chosen). Therefore, the whole sequence converges uniformly on compact subsets of $\widehat{\Omega}$ to this function. This is equivalent to \eqref{CoroDerAsympSqrt_n-2}.

Finally, expanding the rational function in the right side of \eqref{CoroDerAsympSqrt_n-3} in partial fractions, we get
$$\frac{L_{n}^{(j+1)}(z)}{n L_{n}^{(j)}(z)}= \frac{1}{n} \sum_{k=1}^{n-j} \frac{1}{z-x_k^{(j)}}= \frac{n-j}{n}\; \int \frac{d\sigma(L_n^{(j)})(x)}{z-x}.$$
Hence, it is straightforward that for each fixed $j\geq 1$, the sequence of rational functions $\{L_{n}^{(j+1)}(z)/n L_{n}^{(j)}(z)\}$ is uniformly bounded on each compact subset of $\widehat{\Omega}$.

As, all the measures $\sigma(L_n^{(j)}),$ are supported in $\ch{\Delta_0 \cup \Delta_1}$ then for a fixed $z \in \widehat{\Omega}$, the function $f_{z}(x)=(z-x)^{-1}$ is continuous on $\ch{\Delta_0 \cup \Delta_1}$. Therefore, from \eqref{Thm-AsinDer-2}, we find that any subsequence of  $\{L_{n}^{(j+1)}(z)/n L_{n}^{(j)}(z)\}$ which converges uniformly on compact subsets of $\widehat{\Omega}$ converges pointwise to
\begin{equation}\label{Right_01}
    \int \frac{d\mu_{\Delta}(x)}{z-x}.
\end{equation}
 Thus, the whole sequence converges uniformly on compact subsets of $\widehat{\Omega}$ to this function as stated in \eqref{CoroDerAsympSqrt_n-3}.  Substituting \eqref{EqulMeasure_2} into \eqref{Right_01}, we obtain the expression in the right side of \eqref{CoroDerAsympSqrt_n-3}.
\end{proof}

To discuss the zeros of Sobolev extremal polynomials $L_n$,   we need to introduce some notation. Recall that  ${\Omega}=\overline{\CC} \setminus \Delta$ and that in the case we are now considering $g_{\Omega}(z,\infty)$ is given by \eqref{GreenFun-1}. For $\rho>0$, let $G_{\rho}$ be the set made up of all the connected components of  $\{z  \in \CC: g_{\Omega}(z,\infty) < \rho \}$ which are disjoint from $\Delta_0$. Set $\dsty G=\bigcup_{\rho >0} G_{\rho}$ and
\[\Gamma=\partial G \cup \Delta_0,\]
where $\partial G $ is the boundary of $G$. In the sequel, $[\mu]_{\Gamma}$ denotes the balayage of a measure $\mu$ onto $\Gamma$.  See \cite[Appendix  A:VII]{StaTo92} for a brief introduction to the notion of balayage of  a measure and for more details we refer  to \cite[Ch. IV]{Lan72}.

\begin{theorem}\label{BalayageTh01} Under the hypotheses of Theorem \ref{ThDerAsympSqrt_n}, let $\sigma$ be a limit  of a subsequence of $\{\sigma{(L_n)}\}$ in the  sense of the weak star  topology of measures. Then   $S(\sigma) \subset \Delta \cup \overline{G}$  and  $[\sigma]_{\Gamma}=[\mu_{\Delta}]_{\Gamma}$.
\end{theorem}

\begin{proof} The proof of this result is similar to that of  \cite[Theorem 2]{GaKu97} with $p=2$. The main tool in the proof of \cite[Theorem 2]{GaKu97} is \cite[Theorem 5 ]{GaKu97} on the distribution of zeros of certain family of  weighted polynomials. For the application of \cite[Theorem 5 ]{GaKu97} it is necessary to proof that  if $z \in \CC \setminus G$ then
\begin{equation}\label{LimSupSqrtPoly-1}
   \limsup_{n\to \infty} {|L_n(z)|^{1/n}}\leq \capa{\Delta} \, e^{g_{\Omega}(z,\infty)}.
\end{equation}
Replacing in the proof of \cite[Lemma 8]{GaKu97} the expressions \cite[(4.1)-(4.2)]{GaKu97} by \ref{Thm-AsinDer-1}, \ref{AsymRootSobNorm-01} and \ref{AsymRootSobNorm-02}, it is straightforward  to deduce  \ref{LimSupSqrtPoly-1}.
 \end{proof}

It is to be expected (and numerical experiments seem to indicate) that the accumulation points of the zeros of the polynomials $L_n$ draw $\Gamma$.

\end{document}